\documentclass[a4paper,oneside,english]{amsart}
\usepackage[T1]{fontenc}
\usepackage[utf8]{inputenc}
\usepackage{mathrsfs}
\usepackage{enumitem}
\usepackage{amstext}
\usepackage{amsthm}
\usepackage{amssymb}
\usepackage{longtable}
\usepackage[tmargin=3cm,bmargin=3cm,lmargin=3cm,rmargin=3cm]{geometry}

\makeatletter

\pdfpageheight\paperheight
\pdfpagewidth\paperwidth

\numberwithin{equation}{section}
\numberwithin{figure}{section}
\theoremstyle{plain}
\newtheorem{thm}{\protect\theoremname}
  \theoremstyle{plain}
  \newtheorem{prop}[thm]{\protect\propositionname}
  \theoremstyle{definition}
  \newtheorem{defn}[thm]{\protect\definitionname}
  \theoremstyle{remark}
  \newtheorem{rem}[thm]{\protect\remarkname}
  \theoremstyle{plain}
  \newtheorem{lem}[thm]{\protect\lemmaname}
  \theoremstyle{plain}
  \newtheorem{cor}[thm]{\protect\corollaryname}
  \theoremstyle{definition}
  \newtheorem{example}[thm]{\protect\examplename}

\usepackage{mathrsfs}
\usepackage[cal=boondoxo,scr=esstix]{mathalfa}

\makeatother

\usepackage{babel}
  \providecommand{\corollaryname}{Corollary}
  \providecommand{\definitionname}{Definition}
  \providecommand{\examplename}{Example}
  \providecommand{\lemmaname}{Lemma}
  \providecommand{\propositionname}{Proposition}
  \providecommand{\remarkname}{Remark}
\providecommand{\theoremname}{Theorem}

\begin{document}
\global\long\def\quot{\operatorname{Quot}}

\global\long\def\trdeg{\operatorname{tr.deg}}

\global\long\def\height{\operatorname{ht}}

\global\long\def\rk{\operatorname{rk}}

\global\long\def\ord{\operatorname{ord}}

\global\long\def\initial{\operatorname{in}}

\global\long\def\gr{\operatorname{gr}}

\global\long\def\lex{\operatorname{lex}}

\global\long\def\supp{\operatorname{supp}}

\global\long\def\assign{:=}

\global\long\def\ll{\mathfrak{L}}

\global\long\def\kk{\mathscr{k}}
\global\long\def\KK{\mathcal{k}}

\title[{The Łojasiewicz Exponent via
The Hamburger-Noether Process}]{The Łojasiewicz Exponent via The Valuative Hamburger-Noether Process}

\author{Szymon Brzostowski, Tomasz Rodak}

\date{29 September 2016}
\begin{abstract}
Let $\kk$ be an algebraically closed field of any characteristic.
We apply the Hamburger-Noether process of successive quadratic transformations
to show the equivalence of two definitions of the Łojasiewicz
exponent $\ll(\mathfrak{a})$ of an ideal $\mathfrak{a}\subset\kk[[x,y]]$. 
\end{abstract}



\maketitle
\tableofcontents{}

\section{Introduction}

Let $\kk$ be an algebraically closed field of arbitrary characteristic.
Let $\Xi$ denote the set of pairs of formal power series $\varphi\in\kk[[t]]^{2}$
such that $\varphi\ne\mathbf{0}$ and $\varphi(0)=\mathbf{0}$. We
call the elements of $\Xi$ \emph{parametrizations}. We say that a
parametrization $\varphi$ is a \emph{parametrization of }a formal
power series $f\in\kk[[x,y]]$ if $f\circ\varphi=0$. For $\varphi=(\varphi_{1},\ldots,\varphi_{n})\in\kk[[t]]^{n}$
we put $\ord\varphi:=\min_{j}\ord\varphi_{j}$, where $\ord\varphi_{j}$
stands for the order of the power series $\varphi_{j}$. Let $\mathfrak{a}\subset\kk[[x,y]]$
be an ideal. We consider the \emph{Łojasiewicz exponent} of $\mathfrak{a}$
defined by the formula 
\begin{equation}
\ll(\mathfrak{a}):=\sup_{\varphi\in\Xi}\left(\inf_{f\in\mathfrak{a}}\frac{\ord f\circ\varphi}{\ord\varphi}\right).\label{eq:1-1}
\end{equation}
Such concept was introduced and studied by many authors in different
contexts. Lejeune-Jalabert and Teissier {\cite{LT08}} observed
that, in the case of several complex variables, $\ll(\mathfrak{a})$
is the optimal exponent $r>0$ in the Łojasiewicz inequality 
\[
\exists_{C,\varepsilon>0}\forall_{||x||<\varepsilon}\max_{j}|f_{j}(x)|\geqslant C||x||^{r},
\]
where $(f_{1},\ldots,f_{k})$ is an arbitrary set of generators of
$\mathfrak{a}$. Moreover, they proved that, with the help of the
notion of integral closure of an ideal, the number $\ll(\mathfrak{a})$
may be seen algebraically. This is what we generalize below (see Theorem
\ref{thm:Our main}) partly answering \cite[Question
2]{Brzostowski2015}. D'Angelo {\cite{D'Angelo82}} introduced
$\ll(\mathfrak{a})$ 
independently, as an order of contact of $\mathfrak{a}$. He showed
that this invariant plays an important role in complex function theory
in domains in $\mathbb{C}^{n}$.

There has been some interest in understanding the nature of the curves
that `compute' $\ll(\mathfrak{a})$. In fact, the supremum in (\ref{eq:1-1})
may be replaced by maximum. A more exact result in this direction
says that if $\mathfrak{a}=(f_{1},\ldots,f_{m})\kk[[x,y]]$ is an
$(x,y)$-primary ideal, then there exists a parametrization $\varphi$
of $f_{1}\times\cdots\times f_{m}$ such that 
\[
\ll(\mathfrak{a})=\inf_{f\in\mathfrak{a}}\frac{\ord f\circ\varphi}{\ord\varphi}.
\]
For holomorphic ideals, this was proved by Chądzyński and Krasiński
{\cite{CK88a}}, and independently by McNeal and N{é}methi {\cite{McNeal2005}}.
The case of ideals in $\kk[[x,y]]$, where $\kk$ is as above, is
due to the authors {\cite{Brzostowski2015}}. De Felipe, Garc{í}a
Barroso, Gwoździewicz and Płoski {\cite{FBGP2016}} gave a shorter
proof of this result; moreover, they answered \cite[Question
1]{Brzostowski2015}, by showing that $\ll(\mathfrak{a})$ is always a
Farey 
number, i.~e. a rational number of the form $N+b/a$, where $N$,
$a$, $b$ are integers such that $0<b<a<N$.

\section{Methods and results}

Once and for all we agree that all the rings considered in the paper
are commutative with unity. Let $\overline{\mathfrak{a}}$ denote
the integral closure of an ideal $\mathfrak{a}$ (see Section \ref{sec:Integral-closure}).
Our main result is
\begin{thm}
\label{thm:Our main}Let $\mathfrak{a}\subset\kk[[x,y]]$ be an ideal.
Then 
\begin{equation}
\ll(\mathfrak{a})=\inf\left\{ \frac{p}{q}:(x,y)^{p}\kk\left[\left[x,y\right]\right]\subset\overline{\mathfrak{a}^{q}}\right\} .\label{eq:gwiazdka}
\end{equation}
\end{thm}
The general idea of the proof is the following. It is easy to see,
that the right hand side of (\ref{eq:gwiazdka}) is equal to 
\[
\sup_{\nu}\frac{\nu\left(\mathfrak{a}\right)}{\nu\left(\left(x,y\right)\kk\left[\left[x,y\right]\right]\right)},
\]
where $\nu$ runs through the set of all rank one discrete valuations
with center $\left(x,y\right)\kk\left[\left[x,y\right]\right]$. This
is a consequence of the well-known valuative criterion of integral
dependence (see Theorem \ref{thm:valuative criterion}). On the other
hand, there is a correspondance between valuations of the field $\kk(C)$
and parametrizations centered at points of a given
irreducible curve $C$ (see \cite[Chapter V \S10]{walker1962}).  A mathematician's
basic instinct, then, lead us to believe that the same reasoning could
be repeated for parametrizations in place of valuations. For this
we need a version of criterion of integral dependence which is based
on parametrizations (well-known in the complex analytic setting).
This is where the Hamburger-Noether process comes in. Namely, if $\left(R,\mathfrak{m}\right)$
is a local regular two-dimensional domain, then using Abhyankar theorem
(Theorem \ref{thm:Main Abhyankar}) we may find for any given valuation $\nu$ with center
$\mathfrak{m}$ a sequence of quadratic transformations of $R$ producing
rings and their associated valuations which, respectively, approximate
the valuation ring of $\nu$ and $\nu$ itself. The aforementioned
valuations, given by the process, are in fact expressible in a quite
explicit form even in the case $R=\kk\left[\left[x_{1},\ldots,x_{n}\right]\right]$
(see Lemmas \ref{lem:parametry indukcja} and \ref{lem: ord =00003D rzad w szeregach});
however, the unique feature of Abhyankar theorem is the `approximation
phenomenon', which for non-divisorial valuations only holds in the
two-dimensional case (cf. Example \ref{exa: przyk=000142ad tr=0000F3jwymiarowy}).
Altogether, the above observations plus the usual valuative criterion
of integral dependence allows us to prove a parametric version of
the criterion over $\kk\left[\left[x,y\right]\right]$.

The structure of the paper is as follows. Sections \ref{sec:Valuations}
and \ref{sec:Integral-closure} are of introductory nature. In Section
\ref{sec:Quadratic-transformation-of} we give detailed description
of the concept of the quadratic transformation of a local regular
domain. This notion was developed and used by Zariski and Abhyankar
in the 50's in the framework of valuation theory and the resolution
of singularities problem. A sequence of successive quadratic transformations
starting from a local regular domain containing an algebraically closed
field leads to an inductive construction called the Hamburger-Noether
process. This is described in Section \ref{sec:Hamburger-Noether-expansion}.
In this setting Hamburger-Noether process may be considered as a generalization
of a classical construction of the normalization of a plane algebroid
curve (see \cite{Cam80,Plo2013}) to the case of valuations \cite{Galindo1994}.
Finally, in Sections \ref{sec:Parametric-criterion-of} and \ref{sec:The-main-result}
we prove the aforementioned parametric criterion of integral dependence
and as a result obtain Theorem \ref{thm:Our main}.

\section{Valuations\label{sec:Valuations}}

An integral domain $V$ is called a \emph{valuation ring} if every
element $x$ of its field of fractions $K$ satisfies
\[
x\notin V\implies1/x\in V.
\]
We say that $V$ is a \emph{valuation ring of $K$}. The set of ideals
of a valuation ring $V$ is totally ordered by inclusion. In particular,
$V$ is a local ring. In general, this ring need not be Noetherian,
nevertheless its finitely generated ideals are necessarily principal.

A \emph{valuation} of a field $K$ is a group homomorphism $\nu\colon K^{*}\to\text{\ensuremath{\Gamma}}$,
where $\Gamma$ is a totally ordered abelian group (written additively),
such that for all $x,y\in K^{*}$, if $x+y\ne0$ then 
\[
\nu\left(x+y\right)\geqslant\min\left\{ \nu\left(x\right),\nu\left(y\right)\right\} .
\]
Occasionally, when convenient, we will extend $\nu$ to $K$ setting
$\nu\left(0\right):=+\infty$. The image of $\nu$ is called the \emph{value
group} of $\nu$ and is denoted $\Gamma_{\nu}$. Set 
\begin{align*}
R_{\nu}:= & \left\{ x\in K:x=0\text{ or }\nu\left(x\right)\geqslant0\right\} ,\\
\mathfrak{m}_{\nu}:= & \left\{ x\in K:x=0\text{ or }\nu\left(x\right)>0\right\} .
\end{align*}
Then $R_{\nu}$ is a valuation ring of $K$ and $\mathfrak{m}_{\nu}$
is its maximal ideal.

Let $\Gamma$ be an ordered abelian group. A subgroup $\Gamma'\subset\Gamma$
is called \emph{isolated} if the relations $0\leqslant\alpha\leqslant\beta$,
$\alpha\in\Gamma$, $\beta\in\Gamma'$ imply $\alpha\in\Gamma'$. The
set of isolated subgroups of $\Gamma$ is totally ordered by inclusion.
The number of proper isolated subgroups of $\Gamma$ is called the
\emph{rank} of $\Gamma$, and written $\rk\Gamma$. If $\nu$ is a
valuation of a field $K$, then we say that $\nu$ is of rank $\rk\nu:=\rk\Gamma_{\nu}$.
It is well known that the rank of $\nu$ is equal to the Krull dimension
of $R_{\nu}$ \cite[VI.4.5 Proposition 5]{BourbakiCA1972}.

If $V$ is a valuation ring of $K$, then there exists a valuation
$\nu$ of $K$ such that $V=R_{\nu}$. If $\nu_{1},\nu_{2}$ are valuations
of $K$ then $R_{\nu_{1}}=R_{\nu_{2}}$ if and only if there exists
an order-preserving group isomorphism $\varphi\colon\Gamma_{\nu_{1}}\to\Gamma_{\nu_{2}}$
satisfying $\nu_{2}=\varphi\circ\nu_{1}$. In such a case we say that
valuations $\nu_{1}$ and $\nu_{2}$ are \emph{equivalent}. 

Let $R$ be an integral domain with field of fractions $K$. The valuation
$\nu$ of $K$ is said to be \emph{centered} on $R$ if $R\subset R_{\nu}$.
In this case the prime ideal $\mathfrak{p}=\mathfrak{m}_{\nu}\cap R$
is called the \emph{center} of $\nu$ on $R$. Quite generally, if
$A\subset B$ is a ring extension, $\mathfrak{q}$ is a prime ideal
of $B$ and $\mathfrak{p}=\mathfrak{q}\cap A$ then we have a natural
monomorphism $A/\mathfrak{p}\hookrightarrow B/\mathfrak{q}$. Consequently,
the residue field of $\mathfrak{p}$, that is the field of fractions
of $A/\mathfrak{p}$, may be considered as a subfield of the residue
field of $\mathfrak{q}$. In this setting we have the following important
dimension inequality due to I. S. Cohen. We write below $\trdeg_{A}B$
for the transcendence degree of the field of fractions of $B$ over
that of $A$, where $A\subset B$ is an extension of integral domains. 
\begin{thm}[{\cite[Theorem 15.5]{Matsumura1989}}]
\label{thm:dimension inequality}Let $A$ be a Noetherian integral
domain, and $B$ an extension ring of $A$ which is an integral domain.
Let $\mathfrak{q}$ be a prime ideal of $B$ and $\mathfrak{p}=\mathfrak{q}\cap A$;
then we have
\[
\height\mathfrak{q}+\trdeg_{A/\mathfrak{p}}B/\mathfrak{q}\leqslant\height\mathfrak{p}+\trdeg_{A}B.
\]
\end{thm}
In what follows we will be interested in the case where $\left(R,\mathfrak{m},\kk\right)$
is a local Noetherian domain with residue field $\kk$ and $\nu$
is a valuation with center $\mathfrak{m}$ on $R$. We set $\trdeg_{\kk}\nu:=\trdeg_{\kk}R_{\nu}/\mathfrak{m}_{\nu}$.
Directly from the above theorem we get:
\begin{prop}
\label{rank inequality}Let $\left(R,\mathfrak{m},\kk\right)$ be
a local Noetherian domain and let $\nu$ be a valuation with center
$\mathfrak{m}$ on $R$. Then
\[
\rk\nu+\trdeg_{\kk}\nu\leqslant\dim R.
\]
In particular, $\trdeg_{\kk}\nu\leqslant\dim R-1$.
\end{prop}
\begin{defn}
Let $\left(R,\mathfrak{m},\kk\right)$ be a local Noetherian domain
and let $\nu$ be a valuation with center $\mathfrak{m}$ on $R$.
If $\trdeg_{\kk}\nu=\dim R-1$ then we say that $\nu$ is \emph{divisorial}
with respect to $R$ (or is a \emph{prime divisor} for $R$). 
\end{defn}

\section{Integral closure of ideals\label{sec:Integral-closure}}

Let $\mathfrak{a}$ be an ideal in a ring $R$. We say that an element
$x\in R$ is \emph{integral over} $\mathfrak{a}$ if there exist $N\geqslant1$
and $a_{1}\in\mathfrak{a},a_{2}\in\mathfrak{a}^{2},\ldots,a_{N}\in\mathfrak{a}^{N}$
such that
\[
x^{N}+a_{1}x^{N-1}+\cdots+a_{N}=0.
\]
The set of elements of $R$ that are integral over $\mathfrak{a}$
is called the \emph{integral closure} of $\mathfrak{a}$ and is denoted
$\overline{\mathfrak{a}}$. It turns out that the integral closure
of an ideal is always an ideal.

Next theorem is the celebrated valuative criterion of integral dependence.
\begin{thm}[{\cite[Proposition 6.8.4]{HS06a}}]
\label{thm:valuative criterion}Let $\mathfrak{a}$ be an ideal in
an integral Noetherian domain $R$. Let $\mathcal{V}$ be the set
of all discrete valuation rings $V$ of rank one between $R$ and
its field of fractions for which the maximal ideal of $V$ contracts
to a maximal ideal of $R$. Then
\[
\overline{\mathfrak{a}}=\bigcap_{V\in\mathcal{V}}\mathfrak{a}V\cap R.
\]
\end{thm}

\section{Quadratic transformation of a ring\label{sec:Quadratic-transformation-of}}
\begin{defn}
Let $\left(R,\mathfrak{m}\right)$ be a local regular domain and let
$x\in\mathfrak{m}\setminus\mathfrak{m}^{2}$. Set $S=R\left[\frac{\mathfrak{m}}{x}\right]$
and let $\mathfrak{p}$ be a prime ideal in $S$ containing $x$.
Then the ring $S_{\mathfrak{p}}$ is called a \emph{(first) quadratic
transform} of $R$. If $\nu$ is a valuation with center $\mathfrak{m}$
on $R$ and $xR_{\nu}=\mathfrak{m}R_{\nu}$ then $S_{\mathfrak{p}}$,
where $\mathfrak{p}:=R\cap\mathfrak{m}_{\nu}$, is called a \emph{(first)
quadratic transform }of $R$ \emph{along }$\nu$.
\end{defn}
\begin{rem}
Keep the notations from the above definition. Then $xS=\mathfrak{m}S$
and for any $k\in\mathbb{N}$, $x^{k}S\cap R=\mathfrak{m}^{k}S\cap R=\mathfrak{m}^{k}$.
Indeed, the equalities $xS=\mathfrak{m}S$, $x^{k}S=\mathfrak{m}^{k}S$
and the inclusion $\mathfrak{m}^{k}\subset\mathfrak{m}^{k}S\cap R$
are clear. Take $r\in\mathfrak{m}^{k}S\cap R$. Then there exist $l\geqslant0$
and $a_{j}\in\mathfrak{m}^{k+j}$, $j=0,\ldots,l$, such that
\[
a_{0}+\frac{a_{1}}{x}+\cdots+\frac{a_{l}}{x^{l}}=r.
\]
Thus $x^{l}r\in\mathfrak{m}^{k+l}$. On the other hand, $\left(R,\mathfrak{m}\right)$
is a local regular domain, hence the associated graded ring $\gr_{\mathfrak{m}}R$
is an integral domain (as isomorphic to the ring o polynomials $\frac{R}{\mathfrak{m}}\left[Y_{1},\ldots,Y_{n}\right]$).
We have $\left(x^{l}+\mathfrak{m}^{l+1}\right)\cdot\left(r+\mathfrak{m}^{k}\right)=x^{l}r+\mathfrak{m}^{k+l}$,
which is zero in $\gr_{\mathfrak{m}}R$. Consequently, since $x^{l}\notin\mathfrak{m}^{l+1}$
we must have $r\in\mathfrak{m}^{k}$. 
\end{rem}
\begin{rem}
It is clear from the definition, that if $\left(T,\mathfrak{n}\right)$
is a quadratic transformation of $\left(R,\mathfrak{m}\right)$ along
$\nu$ then $\nu$ has center $\mathfrak{n}$ on $T$.
\end{rem}
\begin{prop}
\label{prop:1}Let $\left(R,\mathfrak{m}\right)$ be a local regular
domain of dimension $n>1$. Set $x_{1},\ldots,x_{n}$ as the generators
of $\mathfrak{m}$. Let $R\left[Y\right]$, where $Y=\left(Y_{2},\ldots,Y_{n}\right)$,
be a polynomial ring in $n-1$ variables over $R$. If $\varphi\colon R\left[Y\right]\to S:=R\left[\frac{x_{2}}{x_{1}},\ldots,\frac{x_{n}}{x_{1}}\right]$
is an $R$-homomorphism given by $\varphi\left(Y_{j}\right):=x_{j}/x_{1}$,
$j=2,\ldots,n$, then $\ker\varphi=\left(x_{1}Y_{2}-x_{2},\ldots,x_{1}Y_{n}-x_{n}\right)R\left[Y\right].$
\end{prop}
\begin{proof}
Take $f\in\ker\varphi$. Using successive divisions with remainder
we may write $f$ in the form
\[
f\left(Y\right)=A_{2}\cdot\left(Y_{2}-\frac{x_{2}}{x_{1}}\right)+\cdots+A_{n}\cdot\left(Y_{n}-\frac{x_{n}}{x_{1}}\right)+B,
\]
where $A_{2},\ldots,A_{n}\in S\left[Y\right]$, $B\in R\left[\frac{x_{2}}{x_{1}},\ldots,\frac{x_{n}}{x_{1}}\right]$.
We must have $B=0$, since $f\in\ker\varphi$. There exists $N$ such
that
\begin{equation}
x_{1}^{N}f\left(Y\right)=A_{2}^{'}\cdot\left(x_{1}Y_{2}-x_{2}\right)+\cdots+A_{n}^{'}\cdot\left(x_{1}Y_{n}-x_{n}\right),\label{eq:1}
\end{equation}
where $A_{2}^{'},\ldots,A_{n}^{'}\in R\left[Y\right]$.

Now, observe that $R\left[\left[Y\right]\right]$ is a regular local
ring of dimension $2n-1$ and $x_{1}Y_{2}-x_{2},\ldots,x_{1}Y_{n}-x_{n},x_{1},Y_{2},\ldots,Y_{n}$
is its regular system of parameters \cite[Theorems 15.4, 19.5]{Matsumura1989}.
Thus $$R\left[\left[Y\right]\right]/\left(x_{1}Y_{2}-x_{2},\ldots,x_{1}Y_{n}-x_{n}\right)$$
is a regular local domain and, consequently $\left(x_{1}Y_{2}-x_{2},\ldots,x_{1}Y_{n}-x_{n}\right)R\left[\left[Y\right]\right]$
is a prime ideal. Thus $\left(x_{1}Y_{2}-x_{2},\ldots,x_{1}Y_{n}-x_{n}\right)R\left[Y\right]$
is also prime. Moreover, this ideal does not contain $x_{1}$ since
$x_{1},\ldots,x_{n}$ minimally generates $\mathfrak{m}$. This and
(\ref{eq:1}) gives $f\in\left(x_{1}Y_{2}-x_{2},\ldots,x_{1}Y_{n}-x_{n}\right)R\left[Y\right]$.
\end{proof}
\begin{prop}
\label{prop:przeksztalcenia kwadratowe}Under the notations from Proposition
\ref{prop:1} we have:

\begin{enumerate}[label=\arabic{enumi})]
\item $S$ is regular,\label{enu:jeden}
\item \label{enu:dwa}if $\mathfrak{p}\subset S$ is a prime ideal containing
$x_{1}$ then $S_{\mathfrak{p}}$ is a regular local ring and 
\[
\trdeg_{R/\mathfrak{m}}S_{\mathfrak{p}}/\mathfrak{p}S_{\mathfrak{p}}=\dim R-\dim S_{\mathfrak{p}},
\]
\item \label{enu:trzy}if $\mathfrak{p}=\mathfrak{m}_{\nu}\cap S$, where
$\nu$ is a valuation with center $\mathfrak{m}$ on $R$ such that
$\nu\left(x_{1}\right)\leqslant\nu\left(x_{j}\right)$, $j=2,\ldots,n$,
then 
\[
\trdeg_{R/\mathfrak{m}}\nu-\trdeg_{S_{\mathfrak{p}}/\mathfrak{p}S_{\mathfrak{p}}}\nu=\dim R-\dim S_{\mathfrak{p}}.
\]
\end{enumerate}
\end{prop}
\begin{proof}
Let $\mathfrak{p}\subset S$ be a prime ideal. We have $R\subset S\subset R_{x_{1}}$,
so $R_{x_{1}}=S_{x_{1}}$. Thus, if $x_{1}\notin\mathfrak{p}$ then
\[
S_{\mathfrak{p}}=\left(S_{x_{1}}\right)_{\mathfrak{p}S_{x_{1}}}=\left(R_{x_{1}}\right)_{\mathfrak{p}R_{x_{1}}}=R_{\mathfrak{p}},
\]
hence $S_{\mathfrak{p}}$ is a regular local ring.

Now, assume that $x_{1}\in\mathfrak{p}$. Let $R\left[Y\right]$,
$Y=\left(Y_{2},\ldots,Y_{n}\right)$, be a polynomial ring. Put $\mathfrak{b}:=\left(x_{1}Y_{2}-x_{2},\ldots,x_{1}Y_{n}-x_{n}\right)R\left[Y\right]$.
We have $S\simeq R\left[Y\right]/\mathfrak{b}$ by Proposition \ref{prop:1}.
Let $\mathfrak{p}^{\star}:=\mathfrak{p}/x_{1}S$, $S^{\star}:=S/x_{1}S$.
Since $\mathfrak{b}\subset\mathfrak{m}R\left[Y\right]$ and $x_{1}S=\mathfrak{m}S$,
\begin{equation}
S^{\star}=\frac{S}{\mathfrak{m}S}\simeq\frac{R\left[Y\right]}{\mathfrak{m}R\left[Y\right]}\simeq\frac{R}{\mathfrak{m}}\left[Y\right].\label{eq:S gwiazdka}
\end{equation}
The ring $S^{\star}$ is regular, as a ring of polynomials over a
field, thus there exist $y_{2},\ldots,y_{k+1}\in S$, such that $\mathfrak{p}^{\star}S_{\mathfrak{p}^{\star}}^{\star}=\left(y_{2},\ldots,y_{k+1}\right)S_{\mathfrak{p}^{\star}}^{\star}$
and $\height\mathfrak{p}^{\star}=k$. Moreover 
\[
\dim S_{\mathfrak{p}}=\height \mathfrak{p}S_{\mathfrak{p}}=\height\mathfrak{p}S=\height\mathfrak{p}^{\star}+1=k+1
\]
and $\mathfrak{p}S_{\mathfrak{p}}=\left(x_{1},y_{2},\ldots,y_{k+1}\right)S_{\mathfrak{p}}$.
Consequently, $S_{\mathfrak{p}}$ is a regular local ring. This proves
\ref{enu:jeden}.

Using the identifications (\ref{eq:S gwiazdka}), we have
\begin{alignat*}{1}
\trdeg_{R/\mathfrak{m}}\frac{S_{\mathfrak{p}}}{\mathfrak{p}S_{\mathfrak{p}}}= & \trdeg_{R/\mathfrak{m}}\left(\frac{\frac{R}{\mathfrak{m}}\left[Y\right]}{\mathfrak{p}^{\star}\frac{R}{\mathfrak{m}}\left[Y\right]}\right)_{0}=\dim\frac{\frac{R}{\mathfrak{m}}\left[Y\right]}{\mathfrak{p}^{\star}\frac{R}{\mathfrak{m}}\left[Y\right]}\\
= & \dim\frac{R}{\mathfrak{m}}\left[Y\right]-\height\mathfrak{p}^{\star}=n-1-k=\dim R-\dim S_{\mathfrak{p}}.
\end{alignat*}
This gives \ref{enu:dwa}.

Since $R/\mathfrak{m}\subset S_{\mathfrak{p}}/\mathfrak{p}S_{\mathfrak{p}}\subset R_{\nu}/\mathfrak{m}_{\nu}$,
the proof of \ref{enu:trzy} follows from \ref{enu:dwa} and from
the equality
\[
\trdeg_{R/\mathfrak{m}}R_{\nu}/\mathfrak{m}_{\nu}=\trdeg_{S_{\mathfrak{p}}/\mathfrak{p}S_{\mathfrak{p}}}R_{\nu}/\mathfrak{m}_{\nu}+\trdeg_{R/\mathfrak{m}}S_{\mathfrak{p}}/\mathfrak{p}S_{\mathfrak{p}}.
\]
\end{proof}
\begin{lem}
\label{lem:wybor x}Let $\left(T,\mathfrak{n}\right)$ be a quadratic
transformation of $R$. Then
\begin{enumerate}[label=\arabic{enumi})]
\item $\mathfrak{n}^{k}\cap R=\mathfrak{m}^{k}$ for any $k\in\mathbb{N},$
\item if $xT=\mathfrak{m}T$ for some $x\in R$, then $x\in\mathfrak{m}\setminus\mathfrak{m}^{2}$
and $T=S_{\mathfrak{p}}$, where $S:=R\left[\frac{\mathfrak{m}}{x}\right]$
and $\mathfrak{p}:=S\cap\mathfrak{n}$.
\end{enumerate}
\end{lem}
\begin{proof}
By the definition of the quadratic transformation there exist $x'\in\mathfrak{m}\setminus\mathfrak{m}^{2}$
and a prime ideal $\mathfrak{p}'$ in $S':=R\left[\frac{\mathfrak{m}}{x'}\right]$
such that $x'\in\mathfrak{p}'$, $T=S'_{\mathfrak{p}'}$, $\mathfrak{n}=\mathfrak{p'}T$. 

We have
\[
\mathfrak{m}^{k}\supset\mathfrak{n}^{k}\cap R=\left(\mathfrak{n}^{k}\cap S'\right)\cap R\supset\mathfrak{p}'^{k}\cap R\supset x'^{k}S'\cap R=\mathfrak{m}^{k}S'\cap R=\mathfrak{m}^{k}.
\]
This gives the first assertion.

For the proof of the second one, observe that $x\in\mathfrak{m}T\cap R\subset\mathfrak{n}\cap R=\mathfrak{m}$.
Moreover, if $x\in\mathfrak{m}^{2}$, then $\mathfrak{m}=xT\cap R\subset\mathfrak{m}^{2}T\cap R\subset\mathfrak{n}^{2}\cap R=\mathfrak{m}^{2}$,
which is a contradiction. Thus $x\in\mathfrak{m}\setminus\mathfrak{m}^{2}$. 

Set $S:=R\left[\frac{\mathfrak{m}}{x}\right]$. Since $xT=\mathfrak{m}T=x'T$,
the element $x/x'$ is invertible in $T$. Hence $S\subset T$. Let
$\mathfrak{p}:=\mathfrak{n}\cap S$. Clearly $S_{\mathfrak{p}}\subset T$.
On the other hand, the localizations $S_{\frac{x'}{x}}$ and $S'_{\frac{x}{x'}}$
are equal; denote them by $Q$. Since $\mathfrak{p}'Q=\mathfrak{n}\cap Q$
and $\mathfrak{p}Q\subset\mathfrak{n}\cap Q$, 
\[
T=S'_{\mathfrak{p}'}=Q_{\mathfrak{p}'Q}=Q_{\mathfrak{n}\cap Q}\subset Q_{\mathfrak{p}Q}=S_{\mathfrak{p}}.
\]
\end{proof}
\begin{defn}
Let $\left(R,\mathfrak{m}\right)$ be a local regular domain and let
$f\in R$, $f\ne0$. Then we write $\ord_{R}f$ for the greatest $l\geqslant0$
such that $f\in\mathfrak{m}^{l}$. As usually, we also put $\ord_{R}0:=+\infty$.
We will call $\ord_{R}$ the \emph{order function} on $R$. Moreover,
for an ideal $\mathfrak{a}\subset R$ we put $\ord_{R}\mathfrak{a}:=\min_{f\in\mathfrak{a}}\ord_{R}f$.
\end{defn}
\begin{cor}
Let $\left(R,\mathfrak{m}\right)$ be a local regular domain. Then
the order function $\ord_{R}$ is a valuation of the field of fractions
of $R$. Moreover, if $x\in\mathfrak{m\setminus\mathfrak{m}}^{2}$,
$S:=R\left[\frac{\mathfrak{m}}{x}\right]$ and $\mathfrak{p}:=xS$,
then $T:=S_{\mathfrak{p}}$ is a valuation ring of the order function
on $R$.
\end{cor}
\begin{proof}
Since as in the proof of Proposition \ref{prop:przeksztalcenia kwadratowe},
$S/xS$ is isomorphic with the ring of polynomials with coefficients
in $R/\mathfrak{m}$, the ideal $xS$ is prime and $\height xS=1$.
Thus, again by Proposition \ref{prop:przeksztalcenia kwadratowe},
$T$ is a local regular one-dimensional domain. Hence it is a discrete
valuation ring of rank one with valuation given by $\ord_{T}$. By
Lemma \ref{lem:wybor x}, $\mathfrak{n}^{r}\cap R=\mathfrak{m}^{r}$,
so $\left(\mathfrak{n}^{r}\setminus\mathfrak{n}^{r+1}\right)\cap R=\mathfrak{m}^{r}\setminus\mathfrak{m}^{r+1}$
and we get that $\ord_{T}$ restricted to $R$ is equal to $\ord_{R}$.
Consequently, $\ord_{R}$ extends to a valuation of the field of fractions
of $R$ with valuation ring equal to $T$.
\end{proof}
From Proposition \ref{prop:przeksztalcenia kwadratowe} we infer that
the quadratic transformation $S_{\mathfrak{p}}$ of $R$ is again
a regular local domain. If $\height\mathfrak{p}>1$ then $\dim S_{\mathfrak{p}}>1$,
thus we may set $R'=S_{\mathfrak{p}}$ and consider a quadratic transformation
of $R'$. This leads to an inductive process, where at each step we
must choose the `center' of the next quadratic transformation. This
process is finite exactly when at some point as the `center' we take
a height one prime ideal. In this case we end up with a discrete valuation
ring of rank one. 

In what follows we will be interested in the situation in which the
above process is driven by a certain valuation $\nu$ with center
$\mathfrak{m}$ on $R$. Here, at each step as the next `center' we
take the ideal $R_{i}\cap\mathfrak{m}_{\nu}$. As a result we get
a sequence (finite or not) of quadratic transformations along $\nu$:
\begin{equation}
R=R_{0}\subset R_{1}\subset\cdots\subset R_{\nu}.\label{eq:sequence}
\end{equation}

\begin{rem}
Actually, the sequence \ref{eq:sequence} is uniquely determined by
the valuation $\nu$. To see this it is enough to check that a local
quadratic transformation $\left(T,\mathfrak{n}\right)$ of $\left(R,\mathfrak{m}\right)$
along $\nu$ is unique. Let $x,x'\in\mathfrak{m}\setminus\mathfrak{m}^{2}$
be such that $xR_{\nu}=\mathfrak{m}R_{\nu}=x'R_{\nu}$. Set $S:=R\left[\frac{\mathfrak{m}}{x}\right]$,
$\mathfrak{p}:=\mathfrak{m}_{\nu}\cap S$, $T:=S_{\mathfrak{p}}$
and similarly $S':=R\left[\frac{\mathfrak{m}}{x'}\right]$, $\mathfrak{p}':=\mathfrak{m}_{\nu}\cap S'$,
$T':=S'_{\mathfrak{p}'}$. Since $x'/x\in S\setminus\mathfrak{p}$,
$x'/x$ is invertible in $T$. Hence $x'T=xT=\mathfrak{n}$ and $S'\subset T$,
where we set $\mathfrak{n}:=\mathfrak{p}T$. Moreover, $\mathfrak{n}\cap S'=\left(\mathfrak{m}_{\nu}\cap T\right)\cap S'=\mathfrak{m}_{\nu}\cap S'=\mathfrak{p}'$.
Thus $T=T'$ by Lemma \ref{lem:wybor x}. 
\end{rem}
\begin{thm}[{\cite[Proposition 3, Lemma 12]{Abh1956}}]
\label{thm:Main Abhyankar} The sequence (\ref{eq:sequence}) is
finite if and only if $\nu$ is a divisorial valuation with respect
to $R$. In this case there exists $m\geqslant1$ such that 
\[
R=R_{0}\subset R_{1}\subset\cdots\subset R_{m-1}\subset R_{m}=R_{\nu}.
\]
Moreover, if $\dim R=2$ and the sequence (\ref{eq:sequence}) is
infinite, then 
\[
R_{\nu}=\bigcup_{i}R_{i}\quad\text{ and }\quad\mathfrak{m}_{\nu}=\bigcup_{i}\mathfrak{m}_{i},
\]
where $\mathfrak{m}_{i}$ stands for the maximal ideal of $R_{i}$.
\end{thm}
\begin{lem}
\label{lem:wniosek z glownego}Let $\left(R,\mathfrak{m}\right)$
be a two-dimensional local regular domain and let $\nu$ be a valuation
with center $\mathfrak{m}$ on $R$. Assume that (\ref{eq:sequence})
is a sequence of quadratic transformations along $\nu$. Let $F\subset R_{\nu}\setminus\left\{ 0\right\} $
be a finite set and let $h\in R_{\nu}\setminus\left\{ 0\right\} $
be such that for every $f\in F$ we have $f/h\in\mathfrak{m_{\nu}}$.
Then there exists $i\geqslant0$ such that $\dim R_{i}=2$ and $\min_{f\in F}\ord_{R_{i}}f>\ord_{R_{i}}h$.
\end{lem}
\begin{proof}
By Theorem \ref{thm:Main Abhyankar} there exists $i$ such that $f/h\in\mathfrak{m}_{i}$
for any $f\in F$. Hence $\min_{f\in F}\ord_{R_{i}}f>\ord_{R_{i}}h$.
Thus, we get the assertion if $\dim R_{i}=2$. So, assume that $\dim R_{i}=1$.
This means that the sequence (\ref{eq:sequence}) is necessarily finite
and $R_{i}=R_{\nu}$ is a valuation ring of $\ord_{R_{i-1}}$. It
follows that $\ord_{R_{i-1}}=\ord_{R_{i}}$. Since $\dim R_{i-1}=2$,
we get the assertion.
\end{proof}

\section{Hamburger-Noether expansion\label{sec:Hamburger-Noether-expansion}}

Let $\left(R,\mathfrak{m}\right)$ be an $n$-dimensional local regular
domain, $n>1$. We will assume in this section that there exists an
algebraically closed field $k\subset R$ such that $k\to R/\mathfrak{m}$
is an isomorphism. 
\begin{lem}
\label{lem:equicharacteristic lemma}Let $\left(T,\mathfrak{n}\right)$
be a quadratic transformation of $R$. Then the following conditions
are equivalent:
\begin{enumerate}[label=\arabic{enumi}.]
\item $\dim T=n$,\label{enu:1}
\item $\trdeg_{R/\mathfrak{m}}T/\mathfrak{n}=0$,\label{enu:2}
\item the natural homomorphism $k\to T/\mathfrak{n}$ is an isomorphism,\label{enu:4}
\item for every regular system of parameters $x_{1},\ldots,x_{n}$ of $R$
there exist $j\in\left\{ 1,\ldots,n\right\} $ and $a_{1},\ldots a_{j-1},a_{j+1},\ldots,a_{n}\in k$
such that 
\[
\frac{x_{1}}{x_{j}}-a_{1},\ldots,\frac{x_{j-1}}{x_{j}}-a_{j-1},x_{j},\frac{x_{j+1}}{x_{j}}-a_{j+1},\ldots,\frac{x_{n}}{x_{j}}-a_{n}
\]
 is a regular system of parameters of $T$.\label{enu:3}
\end{enumerate}
\end{lem}
\begin{proof}
\ref{enu:1}$\implies$\ref{enu:2} Follows from Proposition \ref{prop:przeksztalcenia kwadratowe}.

\ref{enu:2}$\implies$\ref{enu:4} By the assumptions the field $R/\mathfrak{m}$
is algebraically closed and the field extension $k=R/\mathfrak{m}\subset T/\mathfrak{n}$
is algebraic. Hence, the last inclusion is in fact equality. Consequently,
the field $k\subset T$ is isomorphic with the residue field of $T$.

\ref{enu:4}$\implies$\ref{enu:3} The ideal $\mathfrak{m}T$ is
principal, hence without loss of generality we may assume that $\mathfrak{m}T=x_{1}T$.
Choose $a_{i}\in k$ as the image of $x_{i}/x_{1}$ in $T/\mathfrak{n}$.
Put $S:=R\left[\frac{\mathfrak{m}}{x_{1}}\right]$, $\mathfrak{p}:=\mathfrak{n}\cap S$.
Then by Lemma \ref{lem:wybor x} we have $T=S_{\mathfrak{p}}$, $\mathfrak{n}=\mathfrak{p}T$.
Every $f\in S$ may be written in the form 
\[
f=f_{0}+A\left(\frac{x_{2}}{x_{1}}-a_{2},\ldots,\frac{x_{n}}{x_{1}}-a_{n}\right),
\]
where $f_{0}\in R$ and $A\in R\left[Y_{2},\ldots,Y_{n}\right]$ is
a polynomial without constant term. We have $f\in\mathfrak{n}$ if
and only if $f_{0}\in\mathfrak{m}$, hence 
\begin{align*}
\mathfrak{p} & =\left(x_{1},\frac{x_{2}}{x_{1}}-a_{2},\ldots,\frac{x_{n}}{x_{1}}-a_{n}\right)R\left[\frac{\mathfrak{m}}{x_{1}}\right].
\end{align*}
Thus 
\[
\frac{\mathfrak{p}}{x_{1}S}\simeq\left(Y_{2}-a_{2},\ldots,Y_{n}-a_{n}\right)\frac{R}{\mathfrak{m}}\left[Y\right],
\]
by Proposition \ref{prop:1}. Consequently $\dim T=\dim S_{\mathfrak{p}}=\height\mathfrak{p}=n$.

\ref{enu:3}$\implies$\ref{enu:1} Obvious.
\end{proof}
\begin{example}
\label{exa: przyk=000142ad tr=0000F3jwymiarowy}Set 
\begin{gather*}
\nu\left(x\right):=\left(0,0,1\right),\\
\nu\left(y\right):=\left(0,1,0\right),\\
\nu\left(z\right):=\left(1,0,0\right)
\end{gather*}
and for any $f\in\kk\left[\left[x,y,z\right]\right]\setminus\left\{ 0\right\} $
put as $\nu\left(f\right)$ the lexicographic minimum of 
\[
\left\{ a\nu\left(x\right)+b\nu\left(y\right)+c\nu\left(z\right):\left(a,b,c\right)\in\supp f\right\} ,
\]
where $\supp f$ denotes the set of $\left(a,b,c\right)\in\mathbb{Z}^{3}$
such that the monomial $x^{a}y^{b}z^{c}$ appears in the expansion
of $f$ with non-zero coefficient. It is easy to see that $\nu$ extends
to a valuation with center $\left(x,y,z\right)\kk\left[\left[x,y,z\right]\right]$.
The value group $\Gamma_{\nu}$ is equal to $\mathbb{Z}^{3}$ with
lexicographical ordering. Let
\[
\kk\left[\left[x,y,z\right]\right]=:R_{0}\subset R_{1}\subset\cdots\subset R_{\nu}
\]
be the sequence of successive quadratic transformations of $\kk\left[\left[x,y,z\right]\right]$
along $\nu$. Observe that $\nu\left(z/y\right)=\left(1,-1,0\right)>\mathbf{0}$,
hence $z/y\in R_{\nu}$. Nevertheless, we claim that $z/y\notin\bigcup_{i=0}^{\infty}R_{i}$.
Indeed, set $S:=R_{0}\left[y/x,z/x\right]$ and notice that, since
$\nu\left(x\right)<\nu\left(y\right)<\nu\left(z\right)$, we have
\[
\mathfrak{p}:=\mathfrak{m}_{\nu}\cap S=\left(x,\frac{y}{x},\frac{z}{x}\right)S
\]
is a maximal ideal in $S$. Thus $R_{1}=\left(R_{0}\right)_{\mathfrak{p}}$
and $x_{1}:=x$, $y_{1}:=y/x$, $z_{1}:=z/x$ is the regular system
of parameters in $R_{1}$, where again $\nu\left(x_{1}\right)<\nu\left(y_{1}\right)<\nu\left(z_{1}\right)$.
Obviously $z/y=z_{1}/y_{1}\notin R_{1}$ and in the same way $z/y\notin R_{2}$
and so on. This proves that the second statement in the Theorem \ref{thm:Main Abhyankar}
does not hold in the multidimensional case. 
\end{example}
\begin{lem}
\label{lem:parametry indukcja}Let $\left(T,\mathfrak{n}\right)$
be an $n$-dimensional local regular domain such that there exists
a sequence 
\begin{equation}
R=R_{0}\subset R_{1}\subset\cdots\subset R_{m}=T,\label{eq:finite sequence}
\end{equation}
where for each $i=1,\ldots,m$, $R_{i}$ is a quadratic transformation
of $R_{i-1}$. Set $x_{1},\ldots,x_{n}$ as the generators of $\mathfrak{m}$.
Then there exists a regular system of parameters $y_{1},\ldots,y_{n}$
of $T$ and polynomials $A_{1},\ldots,A_{n}\in k\left[Y_{1},\ldots,Y_{n}\right]$
such that $x_{j}=A_{j}\left(y_{1},\ldots,y_{n}\right)$, $j=1,\ldots,n$.
\end{lem}
\begin{proof}
Induction with respect to $m$. The case $m=0$ is trivial. Assume
that the assertion is true for some $m-1\geqslant0$. Consider the
sequence (\ref{eq:finite sequence}). By Proposition \ref{prop:przeksztalcenia kwadratowe}
we have $\dim R_{0}\geqslant\dim R_{1}\geqslant\cdots\geqslant\dim R_{m}$.
Thus, for each $i=0,\ldots,m$, $\dim R_{i}=n$. By the induction
hypothesis there exist a regular system of parameters $y_{1}',\ldots,y'_{n}$
of $R_{m-1}$ and polynomials $A'_{1},\ldots,A'_{n}\in k\left[Y_{1},\ldots,Y_{n}\right]$
such that $x_{j}=A'_{j}\left(y'_{1},\ldots,y'_{n}\right)$, $j=1,\ldots,n$.
On the other hand, by Lemma \ref{lem:equicharacteristic lemma}, there
exist $j_{0}$, a regular system of parameters $y_{1},\ldots,y_{n}$
of $R_{m}$ and $a_{1},\ldots a_{j_{0}-1},a_{j_{0}+1},\ldots,a_{n}\in k$
such that 
\begin{align*}
y'_{1}= & y_{j_{0}}\left(y_{1}+a_{1}\right),\\
\vdots\\
y'_{j_{0}-1}= & y_{j_{0}}\left(y_{j_{0}-1}+a_{j_{0}-1}\right),\\
y'_{j_{0}}= & y_{j_{0}},\\
y'_{j_{0}+1}= & y_{j_{0}}\left(y_{j_{0}+1}+a_{j_{0}+1}\right),\\
\vdots\\
y'_{n}= & y_{j_{0}}\left(y_{n}+a_{n}\right).
\end{align*}
Now, according to the above equalities we may easily define polynomials
$A_{1},\ldots,A_{n}$.
\end{proof}
Let $R:=k\left[\left[x_{1},\ldots,x_{n}\right]\right]$ be the ring
of formal power series and let $f\in R\setminus\left\{ 0\right\} $.
We will write $\initial f$ for the \emph{initial form }of $f$, which
is the lowest degree non-zero homogeneous form in the expansion of
$f$. Clearly, $\ord_{R}f$ is equal to the degree of the initial
form of $f$. For the ring of formal power series $R$ as above we
will often write $\ord_{\left(x_{1},\ldots,x_{n}\right)}$ instead
of $\ord_{R}$. 
\begin{lem}
\label{lem: ord =00003D rzad w szeregach}Let $R:=k\left[\left[x_{1},\ldots,x_{n}\right]\right]$
be a ring of formal power series. Let $\left(T,\mathfrak{n}\right)$
be an $n$-dimensional local regular domain between $R$ and field
of fractions of $R$. Assume that there exists a regular system of
parameters $y_{1},\ldots,y_{n}$ of $T$ and polynomials $A_{1},\ldots,A_{n}\in k\left[Y_{1},\ldots,Y_{n}\right]$
such that $x_{j}=A_{j}\left(y_{1},\ldots,y_{n}\right)$, $j=1,\ldots,n$.
Then for every non-zero $f\in R$ we have
\[
\ord_{T}f=\ord_{\left(Y_{1},\ldots,Y_{n}\right)}f\left(A_{1}\left(Y_{1},\ldots,Y_{n}\right),\ldots,A_{n}\left(Y_{1},\ldots,Y_{n}\right)\right).
\]
\end{lem}
\begin{proof}
Set $\Phi:=\left(A_{1},\ldots,A_{n}\right)$. Take $f\in R$, $f\ne0$. 

First, assume that $f$ is a polynomial. We have
\[
f\left(x_{1},\ldots,x_{n}\right)=f\left(\Phi\left(Y_{1},\ldots,Y_{n}\right)\right)_{|Y_{1}=y_{1},\ldots,Y_{n}=y_{n}}.
\]
Thus $f\left(\Phi\left(Y_{1},\ldots,Y_{n}\right)\right)$ is a non-zero
polynomial. Let $P:=\initial f\left(\Phi\left(Y_{1},\ldots,Y_{n}\right)\right)$.
Since $y_{1},\ldots,y_{n}$ is a regular system of parameters of $T$,
\[
\ord_{T}f=\ord_{T}P\left(y_{1},\ldots,y_{n}\right)=\deg P=\ord_{\left(Y_{1},\ldots,Y_{n}\right)}f\left(\Phi\left(Y_{1},\ldots,Y_{n}\right)\right),
\]
 which gives the assertion in this case.

If $f$ is an arbitrary non-zero power series then, cutting  the tail
in the power series expansion of $f$, we find a polynomial $\tilde{f}\in R$
such that $\ord_{T}f=\ord_{T}\tilde{f}$ and $\ord_{\left(Y_{1},\ldots,Y_{n}\right)}f=\ord_{\left(Y_{1},\ldots,Y_{n}\right)}\tilde{f}$.
By the case considered above we have $\ord_{T}\tilde{f}=\ord_{\left(Y_{1},\ldots,Y_{n}\right)}\tilde{f}$. 
\end{proof}

\section{Parametric criterion of integral dependence\label{sec:Parametric-criterion-of}}

Let $R=k\left[\left[x,y\right]\right]$, $\Delta=k\left[\left[t\right]\right]$
be the rings of formal power series over an algebraically closed field
$k$. Let $\mathfrak{m}$ and $\mathfrak{d}$ be the maximal ideals
of $R$ and $\Delta$ respectively. For any $\varphi\in\mathfrak{d}\times\mathfrak{d}$
we have a natural local $k$-homomorphism $\varphi^{*}\colon R\to\Delta$
given by the substitution.
\begin{thm}
\label{thm:integrality criterion}Let $\mathfrak{a}$ be an ideal
in $R$ and let $h\in R.$ Then $h$ is integral over $\mathfrak{a}$
if and only if $\varphi^{*}h\in\varphi^{*}\mathfrak{a}$ for any $\varphi\in\mathfrak{d}\times\mathfrak{d}$.
\end{thm}
\begin{proof}
Assume that $h$ is integral over $\mathfrak{a}$. There exist an
integer $N$ and the elements $a_{j}\in\mathfrak{a}^{j}$, $j=1,\ldots,N$,
such that
\[
h^{N}+a_{1}h^{N-1}+\cdots+a_{N}=0.
\]
 Take parametrization $\varphi\in\mathfrak{d}^{2}$. Let $r:=\ord_{R}\mathfrak{a}$.
Then
\[
N\mathrm{ord}_{\Delta}\varphi^{*}h\geqslant\min_{j}\left(rj+\left(N-j\right)\mathrm{ord}_{\Delta}\varphi^{*}h\right).
\]
 This gives $\mathrm{ord}_{\Delta}\varphi^{*}h\geqslant r$, hence
$\varphi^{*}h\in\varphi^{*}\mathfrak{a}$. 

Assume now, that $h$ is not integral over $\mathfrak{a}$. Since
the case $\mathfrak{a}=0$ is clear, in what follows we will assume
that $\mathfrak{a}\ne0$. By the valuative criterion of integral dependence
(Theorem \ref{thm:valuative criterion}) there exists a valuation
$\nu$ with center $\mathfrak{m}$ on $R$ such that $h\notin\mathfrak{a}R_{\nu}$.
Consider the sequence of successive quadratic transformations of $R$
along $\nu$: 
\[
R=R_{0}\subset R_{1}\subset\cdots\subset R_{\nu}.
\]
Denote by $\mathfrak{m}_{i}$ the only maximal ideal of $R_{i}$,
$i\geqslant0$. Let $F\subset R\setminus\left\{ 0\right\} $ be any
finite set of generators of $\mathfrak{a}$. Then $f/h\in\mathfrak{m}_{\nu}$
for any $f\in F$. Hence, by Lemma \ref{lem:wniosek z glownego} there
exists $i\geqslant0$ such that $\dim R_{i}=2$ and $\min_{f\in F}\ord_{R_{i}}f>\ord_{R_{i}}h$.
By Lemmas \ref{lem:parametry indukcja} and \ref{lem: ord =00003D rzad w szeregach},
there exist polynomials $A,B\in k\left[X,Y\right]$ such that for
any $g\in R$, $\ord_{R_{i}}g=\ord_{\left(X,Y\right)}g\left(A\left(X,Y\right),B\left(X,Y\right)\right)$.
Set $P_{g}\left(X,Y\right):=\initial g\left(A\left(X,Y\right),B\left(X,Y\right)\right)$
for $g\in R$. Then $\deg P_{g}=\ord_{R_{i}}g$. Let $\left(a,b\right)\in k^{2}$
be such that $P_{h}\left(a,b\right)\ne0$ and $P_{f}\left(a,b\right)\ne0$
for $f\in F$. Put $\varphi:=\left(A\left(at,bt\right),B\left(at,bt\right)\right)$.
Clearly $\ord_{\Delta}\varphi^{*}h=\deg P_{h}$ and $\ord_{\Delta}\varphi^{*}f=\deg P_{f}$
for $f\in F$. Hence $\ord_{\Delta}\varphi^{*}h<\min_{f\in F}\ord_{\Delta}\varphi^{*}f=\min_{f\in\mathfrak{a}}\ord_{\Delta}\varphi^{*}f$,
so $\varphi^{*}h\notin\varphi^{*}\mathfrak{a}$. 
\end{proof}

\begin{example}
Let $R=\kk[[x,y]]$, where $\kk$ is an algebraically closed field.
Consider $\mathfrak{a}\assign(x^{2}+y^{3},x^{3})$, $h\assign y^{4}$,
$f\assign x^{2}+y^{3}$. Let $\varphi\assign(t^{3},-t^{2})\in\mathfrak{d}\times\mathfrak{d}$.
Notice that $\varphi^{\ast}f=0$. Now, for any $g\in R\setminus\{0\}$
we define $\nu(g)\assign\left(k,\ord_{\Delta}\varphi^{\ast}g'\right)$,
where $g=f^{k}g'$ and $\gcd(f,g')=1$. It is easy to check that $\nu$
extends to a valuation with center $(x,y)R$ on $R$. We will find
the Hamburger-Noether expansion along $\nu$. Using this we will show
that $h$ is not integral over $\mathfrak{a}$.
\begin{description}
\item [{First~step.}] \begin{flushleft}
We have $\nu(x)=(0,3)$, $\nu(y)=(0,2)$, so we
put $x_{1}\assign\frac{x}{y}$, $y_{1}\assign y$.
\par\end{flushleft}
\item [{Second~step.}] \begin{flushleft}
Now $\nu(x_{1})=(0,1)$, $\nu(y_{1})=(0,2)$, so
let $x_{2}\assign x_{1}$, $y_{2}:=\frac{y_{1}}{x_{1}}$.
\par\end{flushleft}

\end{description}
\begin{flushleft}
Continuing in the above manner we get 
\par\end{flushleft}

\begin{tabular}{c p{12cm}}
\centering
{\small{}}%
\begin{tabular}{|p{1.75cm}|p{2.5cm}|p{2.5cm}|p{3cm}|}
\hline 
{\small{}Recursive formula for $x_{i}$, $y_{i}$ } & \multicolumn{1}{c|}{\small{}Valuation}  & \multicolumn{1}{c|}{\small{}$x_{i}$, $y_{i}$ in terms of $x$, $y$} & {\small{}$x$, $y$ in terms of $x_{i}$, $y_{i}$}\tabularnewline
\hline 
{\small{}$x_{1}\assign\tfrac{x}{y}$,  \newline$y_{1}\assign y$ } & {\small{}$\nu(x_{1})=(0,1)$, \newline$\nu(y_{1})=(0,2)$ } & {\small{}$x_{1}=\tfrac{x}{y}$,\newline $y_{1}=y$ } & {\small{}$x=x_{1}y_{1}$,\newline $y=y_{1}$}\tabularnewline
\hline 
{\small{}$x_{2}\assign x_{1}$, \newline$y_{2}\assign\tfrac{y_{1}}{x_{1}}$ } & {\small{}$\nu(x_{2})=(0,1)$,\newline $\nu(y_{2})=(0,1)$ } & {\small{}$x_{2}=\tfrac{x}{y}$,\newline $y_{2}=\tfrac{y^{2}}{x}$ } & {\small{}$x=x_{2}^{2}y_{2}$, \newline$y=x_{2}y_{2}$}\tabularnewline
\hline 
{\small{}$x_{3}\assign x_{2}$, \newline$y_{3}\assign\tfrac{y_{2}}{x_{2}}+1$ } & {\small{}$\nu(x_{3})=(0,1)$,\newline $\nu(y_{3})=(1,-6)$ } & {\small{}$x_{3}=\tfrac{x}{y}$,\newline $y_{3}=\tfrac{y^{3}+x^{2}}{x^{2}}$ } & {\small{}$x=x_{3}^{3}(y_{3}-1)$, \newline$y=x_{3}^{2}(y_{3}-1)$}\tabularnewline
\hline 
{\small{}$x_{4}\assign x_{3}$, \newline$y_{4}\assign\tfrac{y_{3}}{x_{3}}$ } & {\small{}$\nu(x_{4})=(0,1)$,\newline $\nu(y_{4})=(1,-7)$ } & {\small{}$x_{4}=\tfrac{x}{y}$,\newline $y_{4}=\tfrac{(y^{3}+x^{2})y}{x^{3}}$ } & {\small{}$x=x_{4}^{3}(x_{4}y_{4}-1)$, \newline$y=x_{4}^{2}(x_{4}y_{4}-1)$}\tabularnewline
\hline 
{\small{}$x_{5}\assign x_{4}$, \newline$y_{5}\assign\tfrac{y_{4}}{x_{4}}$ } & {\small{}$\nu(x_{5})=(0,1)$,\newline $\nu(y_{5})=(1,-8)$ } & {\small{}$x_{5}=\tfrac{x}{y}$,\newline $y_{5}=\tfrac{(y^{3}+x^{2})y^{2}}{x^{4}}$ } & {\small{}$x=x_{5}^{3}(x_{5}^{2}y_{5}-1)$,\newline $y=x_{5}^{2}(x_{5}^{2}y_{5}-1)$}\tabularnewline
\hline 
{\small{}$x_{6}\assign x_{5}$, \newline$y_{6}\assign\tfrac{y_{5}}{x_{5}}$ } & {\small{}$\nu(x_{6})=(0,1)$,\newline $\nu(y_{6})=(1,-9)$ } & {\small{}$x_{6}=\tfrac{x}{y}$,\newline $y_{6}=\tfrac{(y^{3}+x^{2})y^{3}}{x^{5}}$ } & {\small{}$x=x_{6}^{3}(x_{6}^{3}y_{6}-1)$, \newline$y=x_{6}^{2}(x_{6}^{3}y_{6}-1)$}\tabularnewline
\hline 
\multicolumn{1}{|c|}{\small{}$\vdots$ } &\multicolumn{1}{c|}{\small{}$\vdots$ } &\multicolumn{1}{c|}{\small{}$\vdots$ } & \multicolumn{1}{c|}{\small{}$\vdots$ }\tabularnewline
\hline 
{\small{}$x_{i}\assign x_{i-1}$, \newline$y_{i}\assign\tfrac{y_{i-1}}{x_{i-1}}$ } & {\small{}$\nu(x_{i})=(0,1)$,\newline $\nu(y_{i})=(1,-i-3)$ } & {\small{}$x_{i}=\tfrac{x}{y}$,\newline $y_{i}=\tfrac{(y^{3}+x^{2})y^{i-3}}{x^{i-1}}$ } & {\small{}$x=x_{i}^{3}(x_{i}^{i-3}y_{i}-1)$,\newline $y=x_{i}^{2}(x_{i}^{i-3}y_{i}-1)$}\tabularnewline
\hline 
\end{tabular}
\tabularnewline
{\centering\small{}Successive steps of the Hamburger-Noether algorithm. }\tabularnewline\tabularnewline
\end{tabular} 

Hence, $\mathfrak{a}R_{i}=(x_{i}^{6}(x_{i}^{i-3}y_{i}-1)^{2}+x_{i}^{6}(x_{i}^{i-3}y_{i}-1)^{3},x_{i}^{9}(x_{i}^{i-3}y_{i}-1)^{3})R_{i}=x_{i}^{9}R_{i}$
and $hR_{i}=x_{i}^{8}R_{i}$ for $i\geqslant6$. Thus $h\notin\bigcup_{i\geqslant6}\mathfrak{a}R_{i}=\mathfrak{a}R_{\nu}$.
Observe also that $y^{5}\in\overline{\mathfrak{a}}\setminus\mathfrak{a}$. 
\end{example}

\section{The main result\label{sec:The-main-result}}

\setcounter{thm}{0}

We keep the notations from the previous section. In particular $R=\kk\left[\left[x,y\right]\right]$,
$\kk$ is algebraically closed and for an ideal $\mathfrak{a}\subset R$
we have

\[
\mathfrak{L}\left(\mathfrak{a}\right)=\sup_{\mathbf{0}\ne\varphi\in\mathfrak{d}\times\mathfrak{d}}\left(\inf_{f\in\mathfrak{a}}\frac{\ord_{\Delta}\varphi^{*}f}{\ord_{\Delta}\varphi^{*}\left(x,y\right)R}\right)=\sup_{\mathbf{0}\ne\varphi\in\mathfrak{d}\times\mathfrak{d}}\frac{\ord_{\Delta}\varphi^{*}\mathfrak{a}}{\ord_{\Delta}\varphi^{*}\left(x,y\right)R}.
\]
Recall that we want to prove the following
\begin{thm}
Let $\mathfrak{a}\subset R$ be an ideal. Then 
\begin{equation}
\ll(\mathfrak{a})=\inf\left\{ \frac{p}{q}:\left(x,y\right)^{p}R\subset\overline{\mathfrak{a}^{q}}\right\} .\label{eq:main equality}
\end{equation}
\end{thm}
\begin{proof}
The cases $\mathfrak{a}=R$ or $\mathfrak{a}=0$ are trivial. Assume
that $\mathfrak{a}$ is a proper ideal and $\height\mathfrak{a}=1$.
Then, clearly, the right hand side of (\ref{eq:main equality}) is
equal to $\infty$. Let $\mathfrak{p}\subset R$ be a height one prime
ideal such that $\mathfrak{a}\subset\mathfrak{p}$. By \cite[Appendix C]{Plo2013}
there exists $f\in R$ such that $\mathfrak{p}=fR$. Hence, one can
find $\varphi\in\mathfrak{d}\times\mathfrak{d}$ such that $\varphi^{*}f=0$
\cite[Theorem 2.1]{Plo2013}. Consequently $\mathfrak{L}\left(\mathfrak{a}\right)=\infty$. 

Now, assume that $\height\mathfrak{a}=2$, so that $\mathfrak{a}$
is $\left(x,y\right)R$-primary.

`$\leqslant$' Fix any $p>0$, $q>0$ such that $\left(x,y\right)^{p}R\subset\overline{\mathfrak{a}^{q}}$.
Take $\varphi\in\mathfrak{d}\times\mathfrak{d}$. Without loss of
generality we may assume that $\ord_{\Delta}\varphi^{*}x\leqslant\ord_{\Delta}\varphi^{*}y$.
Since $x^{p}\in\overline{\mathfrak{a}^{q}}$, Theorem \ref{thm:integrality criterion}
asserts that $\ord_{\Delta}\varphi^{*}x^{p}\geqslant\ord_{\Delta}\varphi^{*}\mathfrak{a}^{q}$.
This easily gives 
\[
\frac{p}{q}\geqslant\frac{\ord_{\Delta}\varphi^{*}\mathfrak{a}}{\ord_{\Delta}\varphi^{*}x}=\frac{\ord_{\Delta}\varphi^{*}\mathfrak{a}}{\ord_{\Delta}\varphi^{*}\left(x,y\right)R}.
\]
Hence $p/q\geqslant\mathfrak{L}\left(\mathfrak{a}\right)$ and consequently
we get the desired inequality.

`$\geqslant$' Take any $p>0$, $q>0$ such that $p/q\geqslant\mathfrak{L}\left(\mathfrak{a}\right)$.
Then, for every $\varphi\in\mathfrak{d}\times\mathfrak{d}$, $\varphi\ne\mathbf{0}$,
we have
\[
\frac{p}{q}\geqslant\frac{\ord_{\Delta}\varphi^{*}\mathfrak{a}}{\ord_{\Delta}\varphi^{*}\left(x,y\right)R}
\]
or, what amounts to the same thing, $\ord_{\Delta}\varphi^{*}\left(x,y\right)^{p}R\geqslant\ord_{\Delta}\varphi^{*}\mathfrak{a}^{q}$.
Hence, for any $h\in\left(x,y\right)^{p}R$ we have $\ord_{\Delta}\varphi^{*}h\geqslant\ord_{\Delta}\varphi^{*}\mathfrak{a}^{q}$.
Thus, $\left(x,y\right)^{p}R\subset\overline{\mathfrak{a}^{q}}$,
by Theorem \ref{thm:integrality criterion}. As a result, we get the
inequality `$\geqslant$' in (\ref{eq:main equality}).
\end{proof}

\vspace{1cm}

\textsc{Szymon Brzostowski}

\noindent\textsc{Faculty of Mathematics and Computer Science, University of Łódź
S. Banacha 22, 90-238 Łódź, Poland}

\emph{E-mail address}: \textbf{brzosts@math.uni.lodz.pl}

\vspace{5mm}

\textsc{Tomasz Rodak}

\noindent\textsc{Faculty of Mathematics and Computer Science, University of Łódź
S. Banacha 22, 90-238 Łódź, Poland}

\emph{E-mail address}: \textbf{rodakt@math.uni.lodz.pl}

\end{document}